\documentclass{amsart}
\usepackage{latexsym,amscd,amssymb,amsopn,amsthm,amsfonts,amsmath}
\usepackage{color}
\usepackage[dvipsnames]{xcolor} 
\copyrightinfo{2010}{American Mathematical Society}
\newtheorem{theorem}{Theorem}[section]

\theoremstyle{definition}

\theoremstyle{remark}

\numberwithin{equation}{section}

\begin{document}

\title[On Grothendieck type duality]
{On the Grothendieck duality for the space of holomorphic Sobolev functions }

\author[A. Levskii]{Arkadii Levskii}
\email{learkasha03@gmail.com}

\author[A. Shlapunov]{Alexander Shlapunov}
\email{ashlapunov@sfu-kras.ru}

\address{Siberian Federal University,
         Institute of Mathematics and Computer Science,
         pr. Svobodnyi 79,
         660041 Krasnoyarsk,
         Russia}



\begin{abstract}
We describe the strong dual space $({\mathcal O}^s (D))^*$ for the space ${\mathcal O}^s (D) =  
 H^s (D) \cap {\mathcal O} (D)$ of holomorphic functions from the Sobolev space $H^s(D)$, $s \in \mathbb Z$, over 
a bounded simply connected plane domain $D$ with infinitely differential boundary $\partial D$. 
We identify the dual space with the space of holomorhic functions on  ${\mathbb C}^n\setminus \overline D$ that 
belong to  $H^{1-s} (G\setminus \overline D)$ for any bounded domain $G$, containing the compact 
$\overline D$, and vanish at the infinity. As a corollary, we obtain a description of the strong dual 
space $({\mathcal O}_F (D))^*$ for the space ${\mathcal O}_F (D)$ of holomorphic functions of 
 finite order of growth in $D$ (here, ${\mathcal O}_F (D)$ is endowed with the inductive limit 
topology with respect to the family of  spaces ${\mathcal O}^s (D)$, $s \in \mathbb Z$). 
 In this way we extend  the classical  Grothendieck-K{\"o}the-Sebasti\~{a}o e Silva 
duality for the space of holomorphic functions.
\end{abstract}

\maketitle

One of the first dualities in the spaces of holomprphic functions was discovered in 
1950-'s independently by A.Grothendieck \cite{Grot2}, 
G. K{\"o}the  \cite{Koth2} and  J. Sebasti\~{a}o e Silva \cite{Silva}, 
who described the strong dual  $({\mathcal O} (D))^*$ for the space of holomorphic functions 
${\mathcal O} (D)$ (endowed with the standard Frech\'et topology) 
in a bounded simply connected domain $D\subset {\mathbb C}$:
\begin{equation} \label{eq.dual.compl}
({\mathcal O} (D))^* \cong {\mathcal O} (\hat {\mathbb C} \setminus D), 
\end{equation} 
where ${\mathcal O} (\hat {\mathbb C} \setminus D)$ is the space of holomorphic 
functions on neighborhoods of the  closed set ${\mathbb C} \setminus D$, vanishing at 
the infinity, endowed with the standard  inductive limit topology of holomorphic functions 
on closed sets. One of the most general results, describing the duality for the spaces 
of solutions to elliptic differential operators with the topology of uniform 
convergence on compact sets,  belong to A. Grothendieck, see  \cite[Theorems 3 and 4]{Grot1}; 
it is similar in a way to \eqref{eq.dual.compl}. Another general scheme of producing dualities for 
(both determined and overdetermined) elliptic systems  was presented in \cite{ShTaDBRK}. It involves the concept of 
Hilbert space with reproducing kernel and the constructed pairings are closely related to the inner 
products of the used Hilbert spaces. However it works easily for formally self-adjoint strongly elliptic 
operators, while in general case the application of the scheme  depends on the 
very subtle information regarding  the properties of the reproducing kernel that is not 
always at hands. Actually, similar results (with the use of classical Bergmann reproducing kernel and pairing 
induced by the inner product of the Lebesgue space $L^2 (D)$) was obtained 
by E. Straube \cite[\S 3]{Str} for harmonic and holomorphic functions of finite order of growth 
of many complex variables. Paper \cite{ShTaDBRK} contains also description of a Grothendick type duality for spaces of 
solutions of finite order of growth to strongly elliptic systems. 

In the present short note we describe a Grothendick 
type duality  for the spaces ${\mathcal O}^s (D)$ of holomorphic Sobolev functions and ${\mathcal O}_F (D)$ of holomorphic functions of finite order of growth over a bounded 
simply connected plane domain $D$ with infinitely differential boundary $\partial D$ 
with the use the pairing induced by the inner product of the Lebesgue space $L^2 (\partial D)$. 

\section{Duality for the space of Sobolev holomorphic functions}
\label{s.1} 

Let $L^2 (D)$ be the Lebesgue space and 
$H^s (D)$, $s \in \mathbb N$, be the Sobolev space of functions over plane domain $D$, endowed with 
the standard inner products. As it is known the scale extends to all values $s>0$, as the  
Sobolev-Slobodetskii scale. We denote by $H^{-s} (D)$, $s >0$, the strong dual for the space 
$H^{s}_0 (D)$ (i.e. for the closure of smooth functions with compact support in $D$ in $H^{s}_0 (D)$); 
the related pairing between elements of  $H^{-s} (D)$ and $H^{s} (D)$ is induced by the inner product 
in the Lebesgue space $L^2 (D)$. Denote by $h(D)$ the space of harmonic functions in $D$,  set  
$h^s(s) = H^s(D) \cap h(D)$ and, similarly, ${\mathcal O}^s(s) = H^s(D) \cap {\mathcal O}(D)$, 
$s\in \mathbb Z$,  where   ${\mathcal O}(D)$ 
the space of holomorhic functions in $D$. By the standard a priori estimates for harmonic functions, 
 $h^s(s)$ and ${\mathcal O}^s(s)$ are closed subspaces of $H^s(D)$, $s\in \mathbb Z$, see, for instance, 
\cite[p. 568]{Str}. We note 
that a holomorphic function is harmonic and then ${\mathcal O}^s (D)$ is a closed subspace 
in $h^s (D)$. According to \cite[Corollary 1.7]{Str}, any element $u \in 
h^s (D)$ has a weak boundary value $u_{|\partial D}$ on $\partial D$ 
belonging to $H^{s-1/2} (\partial D)$, $s\in \mathbb Z$. Of course, $u_{|\partial D}$ coincides  
with the usual trace of $u$ on $\partial D$ if $s\in \mathbb N$.  
It tollows immediately from \cite[Corollary 1.7]{Str}
that for each $u\in {h}^s (D)$ the 
functional $\|u_{|\partial D} \|_{H^{s-1/2} (\partial D)}$ defines a norm on 
$h^s (D)$ that is equivalent to the standard one. As 
${\mathcal O}^s (D) \subset h^s (D)$, we  prefer to endow 
${\mathcal O}^s (D)$ with the norm $\|u_{|\partial D} \|_{H^{s-1/2} (\partial D)}$. 

In  any case,  ${\mathcal O}^s(s)$, $s\in \mathbb Z$,  is a Hilbert space (because 
$\|\cdot \|_{H^{s-1/2} (\partial D)}$ posesses parallelogram property) and we immediately have the standard 
Riesz duality with the pairing related to the corresponding inner product:
\begin{equation} \label{eq.dual.Riesz}
({\mathcal O}^s (D))^* \cong {\mathcal O}^s (D).
\end{equation}
However we want to produce a Grothendieck type duality for ${\mathcal O}^s (D)$. With this purpose, denote by 
$ {\mathcal O}^s (\hat {\mathbb C} \setminus \overline D)$, $s\in \mathbb Z$, the space of 
holomorphic functions in ${\mathbb C} \setminus \overline D$ vanishing at the  infinity that 
belong to  $H^s (G\setminus \overline D)$ for any bounded domain $G$, containing the compact 
$\overline D$.  By the discussion above, any element $v \in 
{\mathcal O}^{s} (\hat {\mathbb C} \setminus \overline D)$ has a weak boundary value $v_{|\partial D}$ 
on $\partial D$ belonging to $H^{s-1/2} (\partial D)$. Then, taking into the account the connection 
between the interior and exterior Dirichlet problems for the Laplace operator, for each 
$v\in {\mathcal O}^{s} (\hat {\mathbb C} \setminus \overline D)$  
functional $\|v_{|\partial D} \|_{H^{s-1/2} (\partial D)}$ defines a norm on 
${\mathcal O}^s (\hat {\mathbb C} \setminus \overline D)$ and, by the discussion above, 
${\mathcal O}^{s} (\hat {\mathbb C} \setminus \overline D)$ is Hilbert space.

\begin{theorem} \label{t.compl.s}
Let $D$ be a bounded simply connected domain with $C^\infty$-smooth boundary. Then for 
each $s\in \mathbb Z$ we have an isomorphism of Banach spaces:
\begin{equation} \label{eq.dual.compl.s}
({\mathcal O}^s (D))^* \cong {\mathcal O}^{1-s} (\hat {\mathbb C} \setminus \overline D).
\end{equation}
\end{theorem}

\begin{proof} We begin with the description of the related pairing. 
First, we note that since $\partial D$ is a compact, then $H^{s'} (\partial D) = H^{s'} _0 (\partial D) $ 
for each $s'\in \mathbb R$. Hence there is a natural duality 
\begin{equation} \label{eq.dual.boundary}
H^{-s'} (\partial D)  \cong H^{s'} (\partial D), \, s'\in {\mathbb R},
\end{equation} 
with the pairing 
$$
\langle \cdot , \cdot \rangle_{\partial D,s'}: H^{-s'} (\partial D)  \times H^{s'} (\partial D) \to {\mathbb C},
$$ 
induced by the inner product in $L^2 (\partial D)$. In particular, 
\begin{equation} \label{eq.pairing.est}
|\langle u , v \rangle_{\partial D,s'} | \leq \|u\|_{H^{s'} (\partial D)} 
\|v\|_{H^{-s'} (\partial D)} 
\mbox{ for all }  v \in H^{-s'} (\partial D), u\in H^{s'} (\partial D).
\end{equation}
For the sake of notations we  drop the index $s'$ in the pairing.

Thus, for each $s\in {\mathbb Z}$ we  obtain a natural 
pairing 
\begin{equation} \label{eq.pairing}
\langle u_{|\partial D} , v_{|\partial D} \rangle_{\partial D}: {\mathcal O}^{s} (D)  \times 
{\mathcal O}^{1-s} (\hat {\mathbb C} \setminus \overline D) \to {\mathbb C},
\end{equation}
inducing a continuous (conjugate-) linear mapping 
\begin{equation} \label{eq.mapping}
{\mathcal O}^{1-s} (\hat {\mathbb C} \setminus \overline D) \ni v \to f_v \in ({\mathcal O}^{s} (D))^*,
\,\, f_v (u) = \langle u_{|\partial D} , v_{|\partial D} \rangle_{\partial D}.
\end{equation}
As \eqref{eq.dual.boundary} is an isomorphism of normed spaces, we see that 
$$\|f_v\|_{({\mathcal O}^{s} (D))^*} = 
\|v_{|\partial D}\|_{H^{s-1/2} (\partial D)}.
$$ 
Now we note that the integral Cauchy formula may be extended to the elements of 
the space ${\mathcal O}^{s} (D)$ with the use of the notion of the weak boundary values.
Namely, for a distribution $u_0 \in H^{s-1/2} (\partial D)$ denote by $K u_0$ its  
 integral Cauchy transform:
$$
(K u_0) (z) = \frac{1}{2\pi \iota }\langle (\cdot-z)^{-1}, \overline u_0  \rangle_{\partial D} , \,\, z \not \in \partial D,
$$
where $\iota $ is the imaginary unit. Of course, $K u_0 (z)$ is just the Cauchy integral for $u_0$ 
if $s\in \mathbb N$. Then for any $u \in {\mathcal O}^{s} (D)$ we have 
\begin{equation}\label{eq.Cauchy.D}
(K u_{|\partial D}) (z)=\left\{
\begin{array}{lll} 
0 , & z\not \in \overline D, \\
u(z) & z \in D;
\end{array}
\right.
\end{equation}
see, for instance, 
\cite{AKy}, \cite{Ch} 
even  for the Martinelli-Bochner integral in ${\mathbb C}^n$. 
Similarly, taking into the account the behaviour at the infinity and the orientation of the curve 
$\partial D$, for elements $v 
\in {\mathcal O}^{1-s} (\hat {\mathbb C} \setminus \overline D)$ we have 
\begin{equation}\label{eq.Cauchy.compl}
- (K v_{|\partial D}) (z)=\left\{
\begin{array}{lll} 
0 , & z\in D, \\
v(z) & z \not \in \overline D.
\end{array}
\right.
\end{equation}
It follows from \eqref{eq.Cauchy.compl}  that if $f_v (u) = 0 $ for all $u \in {\mathcal O}^{s} (D)$ 
then, as the kernel $(\zeta-z)^{-1}$ is holomorphic in $D$ with respect to $\zeta$ for all $z \not \in 
\overline D$,  we conclude that 
$$
0= \langle (\cdot - z )^{-1} , v_{|\partial D} \rangle_{\partial D} =2\pi \iota \, 
(K v_{|\partial D}) (z) = 2\pi \iota \, v(z) 
\mbox{ for all } z \not \in \overline D,
$$
 i.e. mapping \eqref{eq.mapping} is injective.

To finish the proof we have to show that mapping \eqref{eq.mapping} is surjective.  As we noted above,
the space ${\mathcal O}^{s} (D)$  can be treated as a closed subset of the Hilbert space $H^{s-1/2} (\partial D)$.
Then, by Hahn-Banach theorem and Riesz theorem on functionals on Hilbert spaces, for any functional 
$f\in ({\mathcal O}^{s} (D))^*$ there is a function $v_0 = v_0 (f)
\in H^{1/2-s} (\partial D)$ such that 
$$
f(u) = \langle u _{|\partial D}, v_0 \rangle_{\partial D} \mbox{ for all } u \in {\mathcal O}^{s} (D).
$$
Next, denote by $(Kv_0)^-$ the restriction of the integral Cauchy transform to $D$ and 
$(Kv_0)^+$ its restriction to ${\mathbb C} \setminus \overline D$. Then theorems on the boundedness 
of potentials, see \cite[\S 2.3.2.5]{RS}, and the structure of the Cauchy kernel 
yield $(Kv_0)^- \in  {\mathcal O}^{1-s} (D)$, $(Kv_0)^+ \in  {\mathcal O}^{1-s} (
\hat {\mathbb C} \setminus \overline  D)$. Now, by the weak jump theorem of the Cauchy transform, 
see \cite{Ch}, 
we have in the sense of weak boundary values:
$$
(Kv_0)^- _{|\partial D} - (Kv_0)^+ _{|\partial D} = v_0 \mbox{ on } \partial D.
$$
Clearly, by the definition of the  weak boundary values and the classical Cauchy theorem, we have 
$$
\langle u _{|\partial D}, (Kv_0)^- _{|\partial D} \rangle_{\partial D} =0\mbox{ for all } u \in {\mathcal O}^{s} (D).
$$
Therefore
$$
f(u) = - \langle u_{|\partial D} , (Kv_0)^+ _{|\partial D} \rangle_{\partial D} 
\mbox{ for all } u \in {\mathcal O}^{s} (D), 
$$
and then mapping \eqref{eq.mapping} is surjective (i.e. $v = v(f) =-(Kv_0)^+ \in 
{\mathcal O}^{1-s} (\hat {\mathbb C} \setminus \overline  D))$.
\end{proof}

\section{Holomorphic functions of finite order of growth}
\label{s.2} 

One says that a  function $u \in h(D)$ has a finite order of growth near $\partial D$, 
if for each point $z_0 \in \partial D$ there are positive numbers $\gamma$, $C$ and $R$ such that 
$$
|u (z) | \leq C |z-z_0|^{-\gamma} \mbox{ for all } z \in D, |z-z_0|<R. 
$$
The space of such functions we denote by $h_F (D)$. E. Straube \cite{Str} proved that  
$$
h_F (D) = \cup_{s\in {\mathbb Z}}\, h^s (D)
$$ 
and hence we may endow the space with the inductive limit topology with respect to the family 
$\{h^s (D)\}_{s\in {\mathbb Z}}$ of Banach spaces, see, for instance, \cite[\S 6]{Shaef}. Again, as 
${\mathcal O} (D) \subset h(D)$,  we obtain 
\begin{equation} \label{eq.space.F} 
{\mathcal O}_F (D) = 
\cup_{s\in {\mathbb Z}} \, {\mathcal O}^s (D) ;
\end{equation}
we endow this space of holomorphic functions 
of finite order of growth near $\partial D$ with the same topology as $h_F (D)$. 
According to \cite[Ch. 4, Exercise 24e]{Shaef}, ${\mathcal O}_F (D)$ is a ${\rm DF}$-space and then 
its dual is expected to be a Fr\'echet space, see \cite[Ch. 4, Exercise 24a]{Shaef}. Thus, we denote by 
${\mathcal O} (\hat {\mathbb C} \setminus \overline D)$ the space of holomorphic functions 
in ${\mathbb C} \setminus \overline D$ vanishing at the infinity. 
By the Sobolev Embedding Theorem, 
\begin{equation} \label{eq.space.FF}
{\mathcal O} (\hat {\mathbb C} \setminus \overline D) \cap C^\infty ({\mathbb C} \setminus  D)= \cap_{s\in \mathbb Z} 
{\mathcal O}^s (\hat {\mathbb C} \setminus \overline D) .
\end{equation}  
We endow the space 
with the projective limit topology with respect to the family 
$\{ {\mathcal O}^{s} (\hat {\mathbb C} \setminus \overline D) \}_{s \in {\mathbb Z}}$ of the Banach spaces, 
see \cite[Ch. I, \S 5]{Str}.  Thus,  
${\mathcal O} (\hat {\mathbb C} \setminus \overline D)$ is a Fr\'echet space, see 
\cite[Ch. II, \S 6, Corollary 1]{Str}.  
 
\begin{theorem} \label{t.compl.F}
Let $D$ be a bounded simply connected domain with $C^\infty$-smooth boundary. Then 
 we have a topological isomorphism: 
\begin{equation} \label{eq.dual.compl.F}
({\mathcal O}_F (D))^* \cong {\mathcal O} (\hat {\mathbb C} \setminus \overline D) 
\cap C^\infty ({\mathbb C} \setminus D).
\end{equation}
\end{theorem}

\begin{proof} It follows almost immediately from Theorem \ref{t.compl.F}. Indeed, 
as $v_{|\partial D} \in C^\infty (\partial D)$ for each ${\mathcal O} (\hat {\mathbb C} \setminus \overline D) 
\cap C^\infty ({\mathbb C} \setminus D)$, formulae \eqref{eq.space.F} and \eqref{eq.space.FF} imply that 
\eqref{eq.pairing} defines a sesquilinear pairing
\begin{equation} \label{eq.pairing.F}
\langle u_{|\partial D} , v_{|\partial D} \rangle_{\partial D}:{\mathcal O}_F (D)  \times 
{\mathcal O} (\hat {\mathbb C} \setminus \overline D) 
\cap C^\infty ({\mathbb C} \setminus D)   \to {\mathbb C}.
\end{equation}
Again, taking into the account the topologies of the space and inequality \eqref{eq.pairing.est}, 
we may define continuous mapping
\begin{equation} \label{eq.mapping.F}
{\mathcal O} (\hat {\mathbb C} \setminus \overline D) 
\cap C^\infty ({\mathbb C} \setminus D) \ni v \to f_v \in ({\mathcal O}_F (D))^*,
\,\, f_v (u) = \langle u_{|\partial D} , v_{|\partial D} \rangle_{\partial D}.
\end{equation}
Its injectivity and surjectivity follow by the same arguments as in the proof of  Theorem \ref{t.compl.F}. 
Finally, the continuity of the inverse mapping follows from the Closed Graph Theotem for Fr\'echet-spaces, 
see \cite[Ch. 3, Theorem 2.3]{Shaef}. 
\end{proof}

Similarly, we obtain the following statement. 

\begin{theorem} \label{t.compl.FF}
Let $D$ be a bounded simply connected domain with $C^\infty$-smooth boundary. Then 
 we have a topological isomorphism: 
\begin{equation} \label{eq.dual.compl.FF}
({\mathcal O} (D)\cap C^\infty (\overline  D))^* \cong {\mathcal O}_F (\hat {\mathbb C} \setminus \overline D) .
\end{equation}
\end{theorem}

\smallskip
{\bf Acknowlegement.} 
The investigation was supported 
by the Krasnoyarsk Mathematical Center and financed by the Ministry of Science and Higher 
Education of the Russian Federation (Agreement No. 075-02-2024-1429).

\end{document}